\documentclass[reqno, 11pt]{amsart}
\usepackage[margin = 1.3 in]{geometry}
\usepackage{amsthm,amsmath,amsfonts}
\usepackage{hyperref}
\usepackage{amssymb}
\usepackage{stmaryrd}
\usepackage{eucal}
\usepackage[libertine]{newtxmath}
\usepackage{tikz-cd} 
\usepackage{amssymb}

\usepackage{mathrsfs}
\usepackage{todonotes}

\usepackage[nameinlink]{cleveref}

\newtheorem{theorem}{Theorem}[section]
\newtheorem{proposition}[theorem]{Proposition}
\newtheorem{lemma}[theorem]{Lemma}

\newtheorem{definition}[theorem]{Definition}

\newtheorem{remark}[theorem]{Remark}
\newtheorem{question}[theorem]{Question}

\DeclareMathOperator{\Hilb}{Hilb}

\DeclareMathOperator{\Pic}{Pic}

\DeclareMathOperator{\rk}{rk}

\DeclareMathOperator{\Gr}{Gr}

\DeclareMathOperator{\codim}{codim}

\author{Ray Shang}
\title{Interpolation for degree 2 Veroneses of odd dimension}

\begin{document}

\begin{abstract}
A classical fact is that through any $d+3$ general points in $d$-dimensional complex projective space $\mathbb{P}_\mathbb{C}^d$, there exists a unique rational normal curve of degree $d$ passing through them. We generalize this by proving the following: when $n$ is odd, for any $\binom{n+2}{2} + n+1$ general points in $\mathbb{P}_\mathbb{C}^{\binom{n+2}{2} - 1}$, there exist at least $2^{n(n-1)}$ degree 2 Veroneses passing through them. This makes substantial progress on a question of Aaron Landesman and Anand Patel, and extends the work of Arthur Coble. 
\end{abstract}

\maketitle

\tableofcontents

\section{Introduction}

The problem of interpolation is one of the oldest questions in mathematics, dating back to at least Euclid, who postulated in his \emph{Elements} that there exists a unique line passing through any two distinct points in the plane. In the 18th century, interpolation began to be studied with the early tools of algebraic geometry. For example, in 1750, Cramer showed that plane curves of degree $n$ pass through $\frac{n(n+3)}{2}$ general points \cite{cramer} and, in 1779, Waring solved the interpolation problem for graphs of polynomial functions \cite{waring}. In its simplest form, the problem of interpolation asks: given a fixed number of general points, does there always exist a projective variety of a given type passing through them? 
\par
The study of interpolation has a wide range of applications. For example, it has applications to Gromov-Witten theory, the slope conjecture, constructing degenerations, and understanding Hilbert schemes, as discussed in \cite[Subsection 1.1]{landesmanDelPezzo}. Interpolation was also used in Eric Larson's resolution of Severi's 1915 Maximal Rank Conjecture \cite{maximalrank}. Outside of mathematics, interpolation is important for the Newton-Cotes method for numerical integration, for Shamir's cryptographic secret sharing protocol \cite{shamir}, and for Reed-Solomon error-correcting codes \cite{reedSolomon}.
\par
In recent years, the literature on the modern study of interpolation has grown tremendously, especially for the interpolation problem of curves. In roughly chronological order, the development of answering the interpolation problem of curves follows the works of Sacchiero in 1980 \cite{sacchiero}, Ellingsrud-Hirschowitz in 1984 \cite{ellingsrudhirschowitz}, Perrin in 1987 \cite{perrin}, Stevens in 1989 and 1996 \cite{stevens} \cite{stevens2}, Ran in 2007 \cite{ran}, Ballico in 2014 \cite{ballico}, Atanasov in 2015 \cite{atansov}, Atanasov-Larson-Yang in 2019 \cite{atansovLarsonYang}, Vogt in 2018 \cite{vogt}, and Larson-Vogt in 2021 and 2023 \cite{larsonVogt2} \cite{larsonVogt}. These developments have led to the first comprehensive answer to the interpolation problem of curves, given by Larson-Vogt \cite[Theorem 1.2]{larsonVogt}. See the introduction of Larson-Vogt (2023) for more details on this development \cite[Introduction]{larsonVogt}. 
\par
Much less is known about interpolation problems for higher-dimensional varieties. To the author's knowledge, the following represents the current literature. In 1922, Arthur Coble showed that degree 2 Veronese surfaces satisfied interpolation \cite{coble}, where a modern exposition can be found in \cite[Remark 5.4]{landesmanDelPezzo}. Weak interpolation was established by David Eisenbud and Sorin Popescu for scrolls of degree $d$ and dimension $k$ with $d \geq 2k-1$ \cite{eisenbudPopescu} in 2000, and interpolation for 2-dimensional scrolls was established by Izzet Coskun \cite{coskun} in 2006. These results were extended and unified by Aaron Landesman's work, which proved that smooth varieties of minimal degree satisfy interpolation \cite[Theorem 1.1]{landesmanMinimalDegree}. Furthermore, Aaron Landesman and Anand Patel showed that all del Pezzo surfaces satisfy weak interpolation \cite[Theorem 1.1]{landesmanDelPezzo}. 

\subsection{Main Result}

Recall that a degree $d$ Veronese variety is the image of an embedding given by the complete linear system $|\mathcal{O}_{\mathbb{P}^n_\mathbb{C}}(d)|$. This paper makes substantial progress on a question of Aaron Landesman and Anand Patel \cite[Question 6 and 7]{landesmanDelPezzo} and adds to the literature on interpolation problems for higher dimensional varieties, specifically for Veronese varieties. It is known that there exists a unique rational normal curve of degree $d$ through any $d+3$ general points in $\mathbb{P}_\mathbb{C}^d$, that there are exactly four degree 2 Veronese surfaces through any 9 general points in $\mathbb{P}_\mathbb{C}^5$, and that there are at least 630 degree 3 Veronese surfaces through any 13 general points in $\mathbb{P}_\mathbb{C}^9$. In this paper, we prove the following.

\begin{theorem}\label{mainResult}
Let $n$ be a positive odd integer. There exist at least $2^{n(n-1)}$ degree 2 Veroneses of dimension $n$ through any $\binom{n+2}{2} + n + 1$ general points in $\mathbb{P}_\mathbb{C}^{\binom{n+2}{2} - 1}$. 
\end{theorem}

The techniques we use are quite different from those used in \cite{landesmanDelPezzo} and \cite{landesmanMinimalDegree}, as we approach our interpolation problem by studying normal bundles. More specifically, as asserted in Proposition \ref{goal}, interpolation for degree 2 Veronese varieties can be proven by demonstrating that holomorphic normal vector fields that vanish at certain points on the Veronese must actually vanish everywhere. Thus, the main focus of this paper is to prove the hypothesis of Proposition \ref{goal}. In order to do this, we employ auxiliary curves satisfying certain properties, and we obtain these auxiliary curves by smoothing certain types of nodal curves which we suggestively call rational normal curve chains. Our use of auxiliary curves is inspired by Arthur Coble's work on interpolation for degree 2 Veronese surfaces. Other works that study interpolation in relation to normal bundles include \cite{atansovLarsonYang} and \cite{larsonVogt}. 

\subsection{Outline and notation}

In Section \ref{sectionInterpolationInGeneral}, we state the formal definition of interpolation for integral projective schemes over an algebraically closed field, and its equivalence to normal bundle interpolation. This leads us to Proposition \ref{goal}, the hypothesis of which is the main focus of this paper, which is found in Section \ref{sectionPlanOfAttack}. In Subsection \ref{subsectionCobleArgument}, we interpret Arthur Coble's work on degree 2 Veronese surfaces in terms of normal vector fields. Then in Subsection \ref{subsectionStrategy}, using Subsection \ref{subsectionCobleArgument} as motivation, we state our strategy of employing certain auxiliary curves to show that the hypothesis of Proposition \ref{goal} holds true. In Sections \ref{sectionRNCChain} and \ref{sectionSmoothingRNCChain} we carry out the strategy, thereby proving Theorem \ref{mainResult}. The very last paragraph of Section \ref{sectionSmoothingRNCChain} proves the enumerative lower bound described in Theorem \ref{mainResult}. We conclude in Section \ref{sectionConclusion} with some remarks and further questions. 
\par
Let us now establish some notation used throughout this paper. In this paper, unless otherwise stated, we always work over the complex numbers and we let $\mathbb{P}_\mathbb{C}^n$ denote complex projective space of dimension $n$. Let $\Hilb_{P(t)}^r$ denote the Hilbert scheme parameterizing closed subschemes of $\mathbb{P}_\mathbb{C}^r$ with Hilbert polynomial $P$. If $X \subset \mathbb{P}_\mathbb{C}^r$ is a variety lying on a unique irreducible component of the Hilbert scheme (which, for example, is the case when $[X]$ is a smooth point of the Hilbert scheme), then let $\Hilb_X$ denote that irreducible component. Let $F\Hilb_{P(t)}^r$, where $P(t) = (P_1(t), \cdots, P_m(t))$ is a $m$-tuple of polynomials, denote the flag Hilbert scheme parameterizing flags of closed subschemes $X_1 \subset \cdots \subset X_m \subset \mathbb{P}_\mathbb{C}^r$ where $X_i$ has Hilbert polynomial $P_i(t)$.
\par
If $Y \subset X \subset \mathbb{P}_\mathbb{C}^r$ are smooth subvarieties, such that $[Y]$ and $[X]$ are smooth points of $\Hilb_Y$ and $\Hilb_X$, then let $P(t) = (P_Y(t), P_X(t))$ be the tuple of their respective Hilbert polynomials. Note $F\Hilb_{P(t)}^r$ comes with a natural projection $F\Hilb^r_{P(t)} \to \Hilb_{P_Y(t)}^r$. In the geometric fiber of this projection over $[Y] \in \Hilb^r_{P_Y(t)}$, if the point $([Y], [X])$ lies on a unique irreducible component, then let $\Hilb_X^Y$ denote that irreducible component. 

\par
Finally, we let $v_{n,2}: \mathbb{P}_\mathbb{C}^n \to \mathbb{P}_\mathbb{C}^{\binom{n+2}{2} - 1}$ denote a degree 2 Veronese embedding, and let $V_{n,2}$ always denote the image of $v_{n,2}$ in $\mathbb{P}_\mathbb{C}^{\binom{n+2}{2} - 1}$. 

\section{Interpolation in general}\label{sectionInterpolationInGeneral}

In this section we recall the formal setup for interpolation problems, and assert and provide references for the fact that, under certain hypotheses, interpolation and vector bundle interpolation for normal bundles are the same. A more complete exposition of the formal aspects of interpolation, as well as proofs of the equivalencies of different ways of proving interpolation, can be found in \cite[Appendix A]{landesmanDelPezzo}.
\par
Interpolation problems, in their simplest form, may be thought of in the following way. Let $H$ be the space of projective varieties of a given type. Suppose we wonder whether there always exists a projective variety of this type passing through $M$ general points in $\mathbb{P}_\mathbb{C}^r$. Then we are wondering whether the projection 
\[
\{ (X, p_1, \cdots, p_M) | X \in H, p_i \in X \} \to \{ (p_1, \cdots, p_M) | p_i \in \mathbb{P}_\mathbb{C}^r \} \cong (\mathbb{P}_\mathbb{C}^r)^M
\]
has dense image. If the image is dense, this would mean for a general collection of $M$ points, we can find a variety of the given type passing through them. The definition of interpolation formalizes this idea. 
\par
Let $X \subset \mathbb{P}_\mathbb{C}^n$ be an integral projective scheme of dimension $k$ lying on a unique irreducible component of the Hilbert scheme parameterizing closed subschemes with the same Hilbert polynomial. Let $\mathcal{H}_X$ denote the irreducible component, taken to have reduced structure, of the Hilbert scheme on which $[X]$ lies. Let $\mathcal{V}_{X}$ denote the universal family over $\mathcal{H}_X$. 

\begin{definition}
     An $m$-tuple of nonnegative integers 
     $\lambda := (\lambda_1, \cdots, \lambda_m)$ is called admissible if:
    \begin{enumerate}
        \item $\lambda_1 \geq \lambda_2 \geq \cdots \geq \lambda_m$,
        \item for all $1 \leq i \leq m$, we have $0 \leq \lambda_i \leq n-k$,
        \item and $\sum_{i=1}^m \lambda_i \leq \dim \mathcal{H}_X$.
    \end{enumerate}
\end{definition}

Let $\Lambda_i \subset \mathbb{P}_\mathbb{C}^n$ be a plane of dimension $n-k-\lambda_i$, for $1 \leq i \leq m$. Define 
 \[
    \Psi := (\mathcal{V}_{\Lambda_1} \times_{\mathbb{P}_\mathbb{C}^n} \mathcal{V}_X) \times_{\mathcal{H}_X} \cdots \times_{\mathcal{H}_X} (\mathcal{V}_{\Lambda_m} \times_{\mathbb{P}_\mathbb{C}^n} \mathcal{V}_X).
    \]
    Noting that $\mathcal{H}_{\Lambda_i} \cong \Gr(n-k-\lambda_i+1, n+1)$, consider the projection 
    \[
    \pi: \Psi \to \mathcal{H}_X \times (\prod_{i=1}^m \Gr(n-k-\lambda_i+1, n+1) ) \times (\mathbb{P}_\mathbb{C}^n)^m \to (\prod_{i=1}^m \Gr(n-k-\lambda_i+1, n+1) ).
    \]
    Define $q$ and $r$ so that $\dim \mathcal{H}_X = q \dot (n-k) + r$ with $0 \leq r < n-k$. 

\begin{definition}\label{interpolationDefinition}
     We say $X$, or $\mathcal{H}_X$, satisfies
    \begin{enumerate}
        \item $\lambda$-interpolation if the projection $\pi$ is surjective, 
        \item weak interpolation if $((n-k)^q)$-interpolation is satisfied,
        \item interpolation if $((n-k)^q, r)$-interpolation is satisfied,
        \item and strong interpolation if $\lambda$-interpolation is satisfied for all admissible $\lambda$.
    \end{enumerate}
\end{definition}

There is also a notion of vector bundle interpolation. 

\begin{definition}\label{vbInterpolationDefinition}

Let $\lambda$ be admissible. Let $E$ be a locally free sheaf on a scheme $X$ with $H^1(X, E) = 0$. Choose points $p_1, \cdots, p_M$ on $X$ and vector subspaces $V_i \subset E|_{p_i}$, for $1 \leq i \leq M$ with $\codim V_i = \lambda_i$. Then define $E'$ so that 
\[
0 \to E' \to E \to \bigoplus_{i=1}^m \frac{E|_{p_i}}{V_i} \to 0
\]
is a short exact sequence of coherent sheaves. We say $E$ satisfies $\lambda$-interpolation if there exist points $p_1, \cdots, p_M$, and subspaces $V_i \subset E|_{p_i}$  as above, such that 
\[
h^0(E) - h^0(E') = \sum_{i=1}^M \lambda_i. 
\]
Write $h^0(E) = q \cdot \rk E + r$ with $0 \leq r < \rk E$. We say $E$ satisfies 
\begin{enumerate}
    \item weak interpolation if it satisfies $((\rk E)^q)$ interpolation,
    \item interpolation if it satisfies $((\rk E)^q, r)$ interpolation,
    \item and strong interpolation if it satisfies $\lambda$-interpolation for all admissible $\lambda$. 
\end{enumerate}
\end{definition}

The following tells us that up to certain hypotheses, proving interpolation in the sense of Definition \ref{interpolationDefinition} is equivalent to showing vector bundle interpolation for the normal bundle $\mathcal{N}_{X/\mathbb{P}_\mathbb{C}^n}$ in the sense of Definition \ref{vbInterpolationDefinition}. 

\begin{theorem}\label{equivalentInterpolation}
    Suppose $X \subset \mathbb{P}_\mathbb{C}^n$ is an integral projective scheme. Furthermore, suppose $H^1(X, \mathcal{N}_{X/\mathbb{P}_\mathbb{C}^n}) = 0$ and $X$ is a local complete intersection. Then the following are equivalent. 
    \begin{enumerate}
        \item $X$ satisfies interpolation.
        \item The map $\pi$ is dominant.
        \item $\mathcal{N}_{X/\mathbb{P}_\mathbb{C}^n}$ satisfies interpolation. 
    \end{enumerate}
\end{theorem}

\begin{remark}
    Assuming the hypotheses of Theorem \ref{equivalentInterpolation}, there are actually 22 different ways of proving (strong) interpolation, as listed in \cite[Theorem A.7]{landesmanDelPezzo}. Theorem \ref{equivalentInterpolation} is a subset of \cite[Theorem A.7]{landesmanDelPezzo}.
\end{remark}

\begin{remark}
In general, Theorem \ref{equivalentInterpolation} holds over algebraically closed fields of characteristic 0, and the characteristic 0 hypothesis is necessary. For example, \cite[Corollary 7.2.9]{landesmanDelPezzo} shows that degree 2 Veronese surfaces over an algebraically closed field of characteristic 2 satisfy interpolation but their normal bundles do not.
\end{remark}

\section{Our plan of attack for degree 2 Veroneses}\label{sectionPlanOfAttack}

In lieu of the previous section, this section explains our strategy for tackling the interpolation problem of degree 2 Veronese varieties. The outline for this section is the following:
\begin{enumerate}
    \item We rephrase our interpolation problem in terms of Veronese normal bundle interpolation.
    \item We review the work of Arthur Coble on interpolation for degree 2 Veronese surfaces to gain insight on the global sections of their normal bundles and the global sections of the restrictions of their normal bundles to elliptic normal sextic curves.
    \item We state our strategy in full. 
\end{enumerate}
With that being said, we begin by rephrasing our interpolation problem in terms of normal bundle interpolation.

\subsection{Veronese normal bundle interpolation}\label{sectionVeroneseNormalBundleInterpolation}

\begin{proposition}\label{goal}
    Let $V_{n,2}$ denote a smooth Veronese variety of dimension $n$ and degree 2. If there exist distinct points $p_i$, for $1 \leq i \leq \binom{n+2}{2} + n + 1$, on $V_{n,2}$ such that 
    \[
    H^0(V_{n,2}, \mathcal{N}_{V_{n,2}/\mathbb{P}_\mathbb{C}^{\binom{n+2}{2} - 1}} \otimes \mathcal{I}_{p_1, \cdots, p_{\binom{n+2}{2} + n + 1}}) = 0,
    \]
    then degree 2 Veronese varieties of dimension $n$ satisfy interpolation. 
\end{proposition}
\begin{proof}
Note that $V_{n,2}$ is a local complete intersection since it is smooth. Furthermore, using the Euler exact sequences on $\mathbb{P}_\mathbb{C}^n$ and $\mathbb{P}_\mathbb{C}^{\binom{n+2}{2} - 1}$ and the dualized conormal short exact sequence for $V_{n,2}$ shows that $H^1(V_{n,2}, \mathcal{N}_{V_{n,2}/\mathbb{P}_\mathbb{C}^{\binom{n+2}{2} - 1}}) = 0$. Then by Theorem \ref{equivalentInterpolation}, interpolation for these Veronese varieties is equivalent to  interpolation for $\mathcal{N}_{V_{n,2}/\mathbb{P}_\mathbb{C}^{\binom{n+2}{2} - 1}}$. 
\par
In addition, the Euler exact sequences on $\mathbb{P}_\mathbb{C}^n$ and $\mathbb{P}_\mathbb{C}^{\binom{n+2}{2} - 1}$ and the dualized conormal short exact sequence for $V_{n,2}$ shows that 
\[
H^0(V_{n,2}, \mathcal{N}_{V_{n,2}/\mathbb{P}_\mathbb{C}^{\binom{n+2}{2} - 1}}) = \binom{n+2}{2}^2 - (n+1)^2.  
\]
The rank of the normal bundle is $\binom{n+2}{2} - n - 1$. Thus, we see that the number of points $M$ in the vector bundle interpolation setup for $\mathcal{N}_{V_{n,2}/\mathbb{P}_\mathbb{C}^{\binom{n+2}{2} - 1}}$, as given in Definition \ref{vbInterpolationDefinition}, should be $\binom{n+2}{2} + n + 1$, and $V_i = \{ 0 \}$ for $1 \leq i \leq \binom{n+2}{2} + n +1$. Then the $E'$ in question, again with respect to the vector bundle interpolation setup as given in  Definition \ref{vbInterpolationDefinition}, is exactly $\mathcal{N}_{V_{n,2}/\mathbb{P}_\mathbb{C}^{\binom{n+2}{2} - 1}} \otimes \mathcal{I}_{p_1, \cdots, p_{\binom{n+2}{2} + n + 1}}$. This can be seen by tensoring the ideal sheaf short exact sequence
\[
0 \to \mathcal{I}_{p_1, \cdots, p_{\binom{n+2}{2} + n + 1}} \to \mathcal{O}_{V_{n,2}} \to \bigoplus_{i=1}^{\binom{n+2}{2} + n+ 1} \mathbb{C}(p_i) \to 0
\]
by $\mathcal{N}_{V_{n,2}/\mathbb{P}_\mathbb{C}^{\binom{n+2}{2} - 1}}$, where the quotient sheaf is supported at the points $p_i$. Altogether, the dimension of the space of global sections of the normal bundle, the rank of the normal bundle, and the number of points required in the interpolation setup imply that to show interpolation for $\mathcal{N}_{V_{n,2}/\mathbb{P}_\mathbb{C}^{\binom{n+2}{2} - 1}}$, it suffices to find distinct points $p_i \in V_{n,2}$ such that
\[
    H^0(V_{n,2}, \mathcal{N}_{V_{n,2}/\mathbb{P}_\mathbb{C}^{\binom{n+2}{2} - 1}} \otimes \mathcal{I}_{p_1, \cdots, p_{\binom{n+2}{2} + n + 1}}) = 0.
\]
\end{proof}

Thus, the main objective of this paper is to show that the hypothesis of Proposition \ref{goal} is true. To motivate our strategy for accomplishing this, we discuss some aspects of Arthur Coble's proof of interpolation for degree 2 Veronese surfaces.

\subsection{Arthur Coble's argument}\label{subsectionCobleArgument}

In 1922, Arthur Coble proved that there were exactly four degree 2 Veronese surfaces
through any 9 general points in $\mathbb{P}_\mathbb{C}^5$. His proof crucially relies on the following lemma. 

\begin{lemma}\cite{coble}\label{lemmaCobleFact}
    There exists a unique elliptic normal sextic (genus 1 and degree 6) curve through any 9 general points in $\mathbb{P}_\mathbb{C}^5$. 
\end{lemma}

For Coble's argument, an elliptic normal sextic acts as an auxiliary curve for the degree 2 Veronese surfaces. This claim is made precise by the following. 

\begin{proposition}\label{2VeroneseSurfaceBijection}
    Fix 9 general points in $\mathbb{P}_\mathbb{C}^5$. Let $C$ be the unique elliptic normal sextic curve passing through those 9 points. There are natural bijections between the following:
    \begin{enumerate}
        \item The set of 2-Veronese surfaces passing through those 9 points. 
        \item The set of 2-Veronese surfaces containing $C$.
        \item The set of degree 3 line bundles $\mathcal{L}$ on $C$ such that $\mathcal{L}^{\otimes 2} \cong \mathcal{O}_C(1)$.
    \end{enumerate}
\end{proposition}
\begin{proof}
    This follows from \cite[Proposition 5.4, Theorem 5.6]{landesmanDelPezzo}. 
\end{proof}

\begin{remark}
    Proposition \ref{2VeroneseSurfaceBijection} implies that there are exactly four 2-Veronese surfaces through any 9 general points in $\mathbb{P}^5_\mathbb{C}$.
\end{remark}

From Proposition \ref{2VeroneseSurfaceBijection} we can deduce information about global sections of $\mathcal{N}_{V_{2,2}/\mathbb{P}_\mathbb{C}^5}$ and $\mathcal{N}_{V_{2,2}/\mathbb{P}_\mathbb{C}^5}|_C$. In general, let $Y \subset X \subset \mathbb{P}_\mathbb{C}^r$ be smooth subvarieties, such that $[Y]$ and $[X]$ are smooth points of $\Hilb_Y$ and $\Hilb_X$. Let $P(t) = (P_Y(t), P_X(t))$ be the tuple of their respective Hilbert polynomials. The flag Hilbert scheme $F\Hilb_{P(t)}^r$ then has tangent space 
\[
T_{([Y],[X])} F\Hilb_{P(t)}^r = \ker[ H^0(X, \mathcal{N}_{X/\mathbb{P}_\mathbb{C}^r}) \oplus H^0(Y, \mathcal{N}_{Y/\mathbb{P}_\mathbb{C}^r}) \to H^0(Y, \mathcal{N}_{X/\mathbb{P}_\mathbb{C}^r}|_Y) ]
\]
at smooth point $([Y],[X])$ by \cite[Remark 4.5.4]{sernesi}. Note $F\Hilb_{P(t)}^r$ comes with a natural projection $F\Hilb^r_{P(t)} \to \Hilb_{P_Y(t)}^r$. Let $\Hilb^Y_X$ denote the unique irreducible component of the geometric fiber of this projection over $[Y] \in \Hilb_{P_Y(t)}^r$ which contains $([Y], [X])$. Then the tangent space of $\Hilb^Y_X$ at $([Y], [X])$ is isomorphic to $H^0(X,\mathcal{N}_{X/\mathbb{P}_\mathbb{C}^r} \otimes \mathcal{I}_Y)$.
\par
Let $V_{2,2}$ be a degree 2 Veronese surface. Let $P_9$ be the subvariety that is the 9 points $p_1, \cdots, p_9 \in \mathbb{P}_\mathbb{C}^5$ in general position, and let $C$ be the unique elliptic normal sextic passing through them. Let $\Pic(C)[2]$ denote the 2-torsion subscheme of $\Pic(C)$. Note set (3) of Proposition \ref{2VeroneseSurfaceBijection} is the underlying closed points of a translate of $\Pic(C)[2]$, which is a reduced scheme of finitely many points. Since Proposition \ref{2VeroneseSurfaceBijection} implies we have a bijection on closed points between $\Hilb_{V_{2,2}}^C$ and a translate of $\Pic(C)[2]$, with the former being connected and the latter being normal, we have an isomorphism of $\mathbb{C}$-schemes. Since a reduced scheme of finitely many points has no tangent vectors, this implies 
\[
H^0(V_{2,2}, \mathcal{N}_{V_{2,2}/\mathbb{P}_\mathbb{C}^5} \otimes \mathcal{I}_{C}) = 0.
\]
A priori we know there is a canonical injection between $\Hilb_{V_{2,2}}^C$ and $\Hilb_{V_{2,2}}^{P_9}$, inducing an injection
\[
H^0(V_{2,2}, \mathcal{N}_{V_{2,2}/\mathbb{P}_\mathbb{C}^5} \otimes \mathcal{I}_{C}) \cong
T_{([P_9], [V_{2,2}])} \Hilb_{V_{2,2}}^C \hookrightarrow T_{([C],[V_{2,2}])} \Hilb_{V_{2,2}}^{P_9} \cong H^0(V_{2,2}, \mathcal{N}_{V_{2,2}/\mathbb{P}_\mathbb{C}^5} \otimes \mathcal{I}_{p_1, \cdots, p_9})
\]
on tangent spaces. However, since degree 2 Veronese surfaces satisfy interpolation, we have 
\[
H^0(V_{2,2}, \mathcal{N}_{V_{2,2}/\mathbb{P}_\mathbb{C}^5} \otimes \mathcal{I}_{C}) \cong H^0(V_{2,2}, \mathcal{N}_{V_{2,2}/\mathbb{P}_\mathbb{C}^5} \otimes \mathcal{I}_{p_1, \cdots, p_9}),
\]
by Proposition \ref{goal}. This isomorphism is equivalent to the vanishing of 
\[
H^0(C, \mathcal{N}_{V_{2,2}/\mathbb{P}_\mathbb{C}^5}|_C \otimes \mathcal{I}_{p_1, \cdots, p_9}).
\]

\subsection{Our strategy}\label{subsectionStrategy}

Recall that to show interpolation for degree 2 Veronese varieties, it suffices to find distinct points $p_i \in V_{n,2}$ for $1 \leq i \leq \binom{n+2}{2} + n + 1$, such that 
 \[
    H^0(V_{n,2}, \mathcal{N}_{V_{n,2}/\mathbb{P}_\mathbb{C}^{\binom{n+2}{2} - 1}} \otimes \mathcal{I}_{p_1, \cdots, p_{\binom{n+2}{2} + n + 1}}) = 0,
    \]
by Proposition \ref{goal}. Our strategy is to find an auxiliary curve $C \subset V_{n,2} \subset \mathbb{P}_\mathbb{C}^{\binom{n+2}{2} - 1}$ satisfying the following properties: 

\begin{enumerate}\label{strategy}
	\item $C$ is a smooth curve of degree $n(n+1)$ and genus $\frac{n(n-1)}{2}$ embedded in $\mathbb{P}_\mathbb{C}^{\binom{n+2}{2} - 1}$ by a complete linear system $|\mathcal{O}_C(1)|$, such that $\mathcal{O}_C(1)$ and all its square roots are non-special very ample line bundles. 
	\item There exist points $p_1, \cdots, p_{\binom{n+2}{2} + n + 1}$ on $C$ such that 
 \[
 H^0(C, \mathcal{N}_{V_{n,2}/\mathbb{P}_\mathbb{C}^{\binom{n+2}{2} - 1}}|_C \otimes \mathcal{I}_{p_1, \cdots, p_{\binom{n+2}{2} + n + 1}}) = 0.
 \]
\end{enumerate}

Let us first explain condition (1). In Coble's argument, as discussed in the previous subsection, the elliptic normal sextic is an auxiliary curve in the sense of Proposition \ref{2VeroneseSurfaceBijection}. In particular, there is a bijection between sets (2) and (3) of Proposition \ref{2VeroneseSurfaceBijection}. We seek an analogous auxiliary curve $C$ for the degree 2 Veroneses of higher dimension. In particular, we want the auxiliary curve $C$ to satisfy certain properties so that there is a bijection between 
\[
\{ \text{degree 2 Veronese varieties containing } C  \} \text{  }  \longleftrightarrow \text{  } \{  \mathcal{L} \in \Pic(C) | \mathcal{L}^{\otimes 2} \cong \mathcal{O}_C(1)  \}.
\]
This is because if this bijection holds and $V_{n,2}$ is a degree 2 Veronese variety of dimension $n$ containing $C$, then by the same argument discussed in Subsection \ref{subsectionCobleArgument}, there would be a scheme-theoretic isomorphism between $\Hilb_{V_{n,2}}^C$ and a translate of $\Pic(C)[2]$ and thus $H^0(V_{n,2}, \mathcal{N}_{V_{n,2}/\mathbb{P}_\mathbb{C}^{\binom{n+2}{2} - 1}} \otimes \mathcal{I}_C) = 0$. In particular, there are conditions that we can impose on a curve $C$ to guarantee that the aforementioned bijection holds. 

\begin{proposition}\label{globalSectionDimensionSpecification}
Suppose $C$ is a smooth curve in $\mathbb{P}_\mathbb{C}^{\binom{n+2}{2}-1}$ and all the square roots of $\mathcal{O}_C(1)$ are very ample. If $\dim H^0(C, \mathcal{O}_C(1)) = \binom{n+2}{2}$  and $\dim H^0(C, \mathcal{L}) =  n+1$ for all square roots $\mathcal{L}$ of $\mathcal{O}_C(1)$, then there is a bijective correspondence between
\[
\{ \text{degree 2 Veronese varieties containing } C  \} \text{  }  \longleftrightarrow \text{  } \{  \mathcal{L} \in \Pic(C) | \mathcal{L}^{\otimes 2} \cong \mathcal{O}_C(1)  \}.
\]

\end{proposition}
\begin{proof}
First we define the map in the forward direction. Suppose we have a degree 2 Veronese variety which contains $C$. Suppose the Veronese is embedded by $v_{n,2}: \mathbb{P}_\mathbb{C}^n \to \mathbb{P}_\mathbb{C}^{\binom{n+2}{2} - 1}$. Then we obtain a line bundle $\mathcal{L} \cong \mathcal{O}_{v_{n,2}^*(C)}(1)$ which squares to $\mathcal{O}_C(1)$. To be clear, $\mathcal{O}_{v_{n,2}^*(C)}(1)$ is the restriction of $\mathcal{O}_{\mathbb{P}_\mathbb{C}^n}(1)$ to $v_{n,2}^*(C)$, while $\mathcal{O}_C(1)$ is the restriction of $\mathcal{O}_{\mathbb{P}_\mathbb{C}^{\binom{n+2}{2} -1}}(1)$ to $C$. 
\par
Now we define the map in the backwards direction. Suppose we have a line bundle $\mathcal{L}$ on $C$ which squares to $\mathcal{O}_C(1)$. By assumption, $\dim H^0(C, \mathcal{L}) = n+1$, and $\dim H^0(C, \mathcal{O}_C(1)) = \binom{n+2}{2}$. Let $|\mathcal{O}_C(1)|$ embed an abstract curve $\tilde{C}$ as $C$ in $\mathbb{P}_\mathbb{C}^{\binom{n+2}{2} - 1}$. Then a complete linear system $|\mathcal{L}|$ embeds $\tilde{C}$ as a curve in $\mathbb{P}_\mathbb{C}^n$, and there is a degree 2 Veronese map whose composition with $|\mathcal{L}|$ is $|\mathcal{O}_C(1)|$. In particular, we obtain a degree 2 Veronese variety which contains $C$. We would like to show that it is unique to the isomorphism class of $\mathcal{L}$. It suffices to show the following: suppose that we have two degree 2 Veronese varieties $X_1$ and $X_2$ in $\mathbb{P}_\mathbb{C}^{\binom{n+2}{2} - 1}$ which contain $C$. Let $v^{(1)}_{n,2} , v^{(2)}_{n,2}: \mathbb{P}_\mathbb{C}^n \to \mathbb{P}_\mathbb{C}^{\binom{n+2}{2} - 1}$ denote
embeddings associated to $X_1$ and $X_2$, respectively, such that $(v_{n,2}^{(1)})^*(C)$ and $(v_{n,2}^{(2)})^*(C)$ 
are projectively equivalent curves in $\mathbb{P}_\mathbb{C}^n$. Then $X_1 = X_2$.
\par
Let us now prove this. Since all Veronese varieties are projectively equivalent, there is an automorphism $\sigma$ of $\mathbb{P}_\mathbb{C}^{\binom{n+2}{2} - 1}$ which sends $X_1$ to $X_2$. Let $C_2 \subset X_2$ denote the image of $C$ under $\sigma$. Since $(v_{n,2}^{(1)})^*(C)$ and $(v_{n,2}^{(2)})^*(C)$ are projectively equivalent curves in $\mathbb{P}_\mathbb{C}^n$ by assumption, there is an automorphism of $\mathbb{P}_\mathbb{C}^n$ sending $(v_{n,2}^{(2)})^*(C_2)$ to $(v_{n,2}^{(2)})^*(C)$. Thus, there is an automorphism $\sigma'$ of $\mathbb{P}_\mathbb{C}^{\binom{n+2}{2} - 1 }$ stabilizing $X_2$ and sending $C_2$ to $C$. Then the images of $\sigma' \circ \sigma \circ v_{n,2}^{(1)}$ and $v_{n,2}^{(2)}$ are both $X_2$, so $\sigma' \circ \sigma$ is an automorphism of $\mathbb{P}_\mathbb{C}^{\binom{n+2}{2} - 1}$ which sends $X_1$ to $X_2$, and is identity on $C$. But the span of $C$ is all of $\mathbb{P}_\mathbb{C}^{\binom{n+2}{2} - 1}$, which implies that $\sigma' \circ \sigma$ must be the identity automorphism of $\mathbb{P}_\mathbb{C}^{\binom{n+2}{2} - 1}$. In particular, this forces $X_1 = X_2$. Altogether, the two maps we have described are inverses to each other, thereby establishing the bijection.
\end{proof}

Hence, if our auxiliary curve $C$ satisfies condition (1), then the Riemann-Roch theorem implies 
\[
\dim H^0(C, \mathcal{O}_C(1)) = \binom{n+2}{2} \text{ and } \dim H^0(C, \mathcal{L}) = n+1
\]
for every square root $\mathcal{L}$ of $\mathcal{O}_C(1)$, and thus
\[
H^0(V_{n,2}, \mathcal{N}_{V_{n,2}/\mathbb{P}_\mathbb{C}^{\binom{n+2}{2} - 1}} \otimes \mathcal{I}_C) = 0. 
\]
Furthermore, if the auxiliary curve satisfies condition (2), then 
\[
	H^0(V_{n,2}, \mathcal{N}_{V_{n,2}/\mathbb{P}_\mathbb{C}^{\binom{n+2}{2} - 1}} \otimes \mathcal{I}_{p_1, \cdots, p_{\binom{n+2}{2} + n + 1}}) \cong H^0(V_{n,2}, \mathcal{N}_{V_{n,2}/\mathbb{P}_\mathbb{C}^{\binom{n+2}{2} - 1}} \otimes \mathcal{I}_C) = 0,
	\]
thereby proving interpolation. Indeed, when $n$ is odd, we can obtain an auxiliary curve satisfying conditions (1) and (2) by smoothing a certain nodal curve which we call a rational normal curve chain.

\section{Rational normal curve chain}\label{sectionRNCChain}

In this section we construct certain nodal curves, which we call rational normal curve chains, and show that they satisfy certain desirable properties that allow us to smooth it to an auxiliary curve satisfying Properties (1) and (2) of Strategy \ref{strategy}. 
\par
Fix a degree 2 Veronese embedding $v_{n,2}: \mathbb{P}_\mathbb{C}^n \to \mathbb{P}_\mathbb{C}^{\binom{n+2}{2}}$, where $n$ is odd and $n \geq 3$. We now construct a degenerate curve in $\mathbb{P}_\mathbb{C}^n$ of degree $\frac{n(n+1)}{2}$ and arithmetic genus $\frac{n(n-1)}{2}$. First, consider $\frac{n+1}{2}$ rational normal curves $R_i$, where $1 \leq i \leq \frac{n+1}{2}$. Note $\frac{n+1}{2}$ is an integer since $n$ is odd. Glue $R_1$ and $R_2$ at some $n+1$ nodal intersections, where the $n+1$ points of intersection are in general position. Note that such a configuration is possible because the space $\Hilb^{(z_1, \cdots, z_{n+1})}_{R}$, which parameterizes rational normal curves passing through $n+1$ fixed points in general position in $\mathbb{P}_\mathbb{C}^n$, has dimension $2n-2$. Then glue $R_2$ and $R_3$ at $n+1$ points of nodal intersection such that the points of intersection are in general position, and are distinct from the previous $n+1$ points of intersection between $R_1$ and $R_2$. Continue this gluing procedure for $R_{i}$ and $R_{i+1}$, where $3 \leq i \leq \frac{n-1}{2}$. We claim that these rational normal curves can be chosen so that they do not intersect anywhere else besides the specified nodal intersections.

\begin{lemma}
    Let $n \geq 3$. Fix points $z_1, \cdots, z_{n+1} \in \mathbb{P}_\mathbb{C}^n$ in general position. In general, two rational normal curves in $\mathbb{P}_\mathbb{C}^n$ intersecting nodally at $z_1, \cdots, z_{n+1}$ do not intersect anywhere else. 
\end{lemma}
\begin{proof}
    Note the space of rational normal curves passing through these $n+1$ points has dimension $2n-2$. Then the locus 
    \[
    \Sigma = \Hilb^{(z_1, \cdots, z_{n+1})}_R \times \Hilb^{(z_1, \cdots, z_{n+1})}_R \setminus \Delta,
    \]
    where $\Delta$ denotes the diagonal, parameterizing pairs of distinct rational normal curves passing through $z_1, \cdots, z_{n+1}$ has dimension $4n-4$. Let 
    \[
    \Phi \subseteq \Sigma \times (\mathbb{P}_\mathbb{C}^n \setminus \{ z_1, \cdots, z_{n+1} \})
    \]
    be the incidence correspondence given by $\{ ([R], [R'], z) | z \in R \cap R' \}$. Note $R \cap R'$ is either empty or a point. Let $U \subseteq \mathbb{P}_\mathbb{C}^n \setminus \{ z_1, \cdots, z_{n+1} \}$ denote the dense open subset of points $z_{n+2}$ such that $\{ z_1, \cdots, z_{n+2} \}$ are points in general position. 
    Then $\Phi$ projects surjectively to $U$, and the fiber over every point $z_{n+2} \in U$ will be the locus 
    \[
    \Hilb_R^{(z_1, \cdots, z_{n+2})} \times \Hilb_R^{(z_1, \cdots, z_{n+2})} \setminus \Delta
    \]
    parameterizing pairs of distinct rational normal curves passing through $z_1, \cdots, z_{n+2}$, which is irreducible of dimension $2n-2$. Since the fibers and the base $U$ are irreducible, we have $\dim \Phi = n +2n-2 = 3n-2$, so the image of $\Phi$ in its projection to $\Sigma$ has dimension at most $3n-2$, while $\dim \Sigma = 4n-4$. Since $n \geq 3$, we obtain our conclusion. 
\end{proof}

If $n = 3$, then we have finished showing that we can pick rational normal curves which do not intersect anywhere else besides at the prescribed nodal intersections. Now suppose $n \geq 5$. We prove the following. 

\begin{lemma}\label{threeRNCnoMoreIntersection}
    In general, three rational normal curves $R_1, R_2,$ and $R_3$, where $R_1$ and $R_2$ are glued nodally at $n+1$ general points and $R_2$ and $R_3$ are glued nodally at different $n+1$ general points, do not intersect anywhere else. 
\end{lemma}
\begin{proof}
    Suppose we fix $R_2$ and points $p_1, \cdots, p_{n+1}$ in general position on $R_2$, and another distinct set of points $ q_1, \cdots, q_{n+1}$ in general position on $R_2$. We know by the previous lemma that, in general, the $R_1$ and $R_3$ we pick to go through $p_1, \cdots, p_{n+1}$ and $q_1, \cdots, q_{n+1}$, respectively, will both not intersect $R_2$ anywhere else. Then it remains to show that, in general, $R_1$ and $R_3$ will not intersect anywhere else.
    \par
    Let $\Sigma_1 = \Hilb^{(p_1, \cdots, p_{n+1})}_{R_2}$ denote the locus of rational normal curves which pass through $p_1, \cdots, p_{n+1}$. Let $\Sigma_3 = \Hilb^{(q_1, \cdots, q_{n+1})}_{R_2}$ denote the locus of rational normal curves which pass through $q_1, \cdots, q_{n+1}$. We have $\dim \Sigma_1 = \dim \Sigma_3 = 2n-2$. Let 
    \[
    \Phi \subseteq \Sigma_1 \times \Sigma_3 \times (\mathbb{P}_\mathbb{C}^n \setminus \{ p_j, q_j \} )
    \]
    be the incidence correspondence given by $\{ (R_1, R_3, p) | R_i \in \Sigma_i, p \in R_1 \cap R_3 \}$. Let $U \subseteq (\mathbb{P}_\mathbb{C}^n \setminus \{ p_j, q_j \} )$ denote the dense open subset of points $x$ such that $\{ p_1, \cdots, p_{n+1}, x \}$ and $\{ q_1, \cdots, q_{n+1}, x \}$ are both collections of points in general position. Then $\Phi$ projects surjectively onto $U$, and the fiber
    of the projection over $x \in U$ will be irreducible of dimension $2n-2$. Then $\dim \Phi = 2n-2 + n = 3n-2$. Then the dimension of the projection of $\Phi$ to $\Sigma_1 \times \Sigma_3$ is at most $3n-2$, while $\dim(\Sigma_1 \times \Sigma_3) = 4n-4$. Thus, in general, $R_1$ and $R_3$ will not intersect anywhere else, and both will not intersect $R_2$ anywhere else. 
\end{proof}

This case of three rational normal curves shows us how to do the inductive step in general.

\begin{lemma}
    Let $n$ be an odd positive integer such that $n \geq 3$. The rational normal curves $R_1, \cdots, R_{\frac{n+1}{2}}$ in the rational normal curve chain in $\mathbb{P}_\mathbb{C}^n$ can be chosen in such a way that there are no other intersections besides the prescribed nodal intersections.
\end{lemma}
\begin{proof}
    We showed the base case with just two rational normal curves. Suppose we have shown that, in general, a chain of $i$ rational normal curves do not intersect anywhere else for $i < \frac{n+1}{2}$. Then we show that the same conclusion holds for a chain of $i+1$ rational normal curves. Pick some chain of $(i-2)$ rational normal curves $R_2, \cdots, R_{i-1}$ which do not intersect anywhere else except at the prescribed nodal intersections. Choose $p_1, \cdots, p_{n+1}$ in general position on $R_2$ which are distinct from the points of $R_2 \cap R_3$, and choose $q_1, \cdots, q_{n+1}$ in general position on $R_{i-1}$, which are distinct from the points of $R_{i-2} \cap R_{i-1}$. 
    \par
    In general, we can pick $R_1$ passing through $p_1, \cdots, p_{n+1}$ so that it does not intersect $R_2, \cdots, R_{i-1}$ anywhere else, and we can pick $R_i$ passing through $q_1, \cdots, q_{n+1}$ so that it does not intersect $R_2, \cdots, R_{i-1}$ anywhere else. Then the proof that $R_1$ and $R_i$ can be chosen so that they do not intersect anywhere else is analogous to the proof given in Lemma \ref{threeRNCnoMoreIntersection}. 
\end{proof}

\begin{proposition}\label{arithmeticGenusRNCChain}
Fix $n$ to be an odd integer greater than 1. A rational normal curve chain $C$, constructed with rational normal curves $R_1, \cdots, R_{\frac{n+1}{2}}$ of degree $n$, in $\mathbb{P}_\mathbb{C}^n$ is of degree $\frac{n(n+1)}{2}$ and arithmetic genus $\frac{n(n-1)}{2}$.
\end{proposition}
\begin{proof}
    The rational normal curve chain $C$ is of degree $\frac{n(n+1)}{2}$ since each of its irreducible components $R_1, \cdots, R_{\frac{n+1}{2}}$ is of degree $n$. The arithmetic genus of $C$ is defined to be
    \[
    p_a(C) := 1 - \chi(\mathcal{O}_C) = 1 - \dim H^0(C, \mathcal{O}_C) + \dim H^1(C, \mathcal{O}_C). 
    \]
    Suppose we have a global section $s \in H^0(C, \mathcal{O}_C)$. Then for $1 \leq j \leq \frac{n+1}{2}$, $s$ restricts to $s|_{R_j} \in H^0(R_j, \mathcal{O}_{R_j})$, which is a constant. Since $\mathcal{O}_C$ is obtained by gluing together the $\mathcal{O}_{R_j}$ over the nodal intersections, the constant $s|_{R_1}$ determines the constant $s|_{R_2}$, which determines the constant $s|_{R_3}$, and so on. Thus, we see that $\dim H^0(C, \mathcal{O}_C) = 1$. To calculate $\dim H^1(C, \mathcal{O}_C)$, consider the short exact sequence 
    \[
    0 \to \mathcal{O}_C \to v_* v^*(\mathcal{O}_C) \to Q \to 0
    \]
    induced by the normalization $v: \bigsqcup_{j=1}^{\frac{n+1}{2}} R_j \to C$. Note $Q$ is a direct sum of skyscraper sheaves supported at the nodal intersection points $\{ z_{j,k} \}$, where $1 \leq j \leq \frac{n-1}{2}$ and $1 \leq k \leq n+1$ and $z_{j,k}$ denotes the $k$-th intersection point between $R_j$ and $R_{j+1}$. The long exact sequence in cohomology induced by this short exact sequence yields 
    \[
    0 \to H^0(C, \mathcal{O}_C) \to \bigoplus_{j=1}^{\frac{n+1}{2}} H^0(R_j, \mathcal{O}_{R_j}) \to H^0(C, Q) \to H^1(C, \mathcal{O}_C) \to 0.
    \]
    Then 
    \[
    \dim H^1(C, \mathcal{O}_C) = \dim H^0(C, Q) - \dim \bigoplus_{j=1}^{\frac{n+1}{2}} H^0(R_j, \mathcal{O}_{R_j}) + \dim H^0(C, \mathcal{O}_C) = \dim H^0(C, Q) - \frac{n+1}{2} + 1.
    \]
    Note that 
     \[
    H^0(C, Q) \cong \bigoplus_{1 \leq j \leq \frac{n-1}{2}, 1 \leq k \leq n+1} \frac{ \mathcal{O}_{R_j}|_{z_{j,k}} \oplus \mathcal{O}_{R_{j+1}}|_{z_{j,k}} }{ \mathcal{O}_C|_{z_{j,k}}} \cong \bigoplus_{1 \leq j \leq \frac{n-1}{2}, 1 \leq k \leq n+1} \frac{\mathcal{O}_{R_j}(z_{j,k}) \oplus \mathcal{O}_{R_{j+1}}(z_{j,k})}{\mathcal{O}_C(z_{j,k})}
    \]
    where $\mathcal{O}_{R_j}(z_{j,k})$ denotes the fiber of the trivial bundle over the point $z_{j,k}$. Then 
    \[
    \dim H^1(C, \mathcal{O}_C) = \frac{(n+1)(n-1)}{2} - \frac{n+1}{2} + 1 = \frac{n(n-1)}{2},
    \]
    which implies our claim. 
\end{proof}

Now that we have shown Proposition \ref{arithmeticGenusRNCChain}, we verify the following, which will allow us to deform a rational normal curve chain to a smooth curve.

\begin{proposition}\label{RNCChainLineBundleNotSpecial}
    Let $C \subset \mathbb{P}_\mathbb{C}^n$ be a rational normal curve chain. Then 
    \[
    H^1(C, \mathcal{O}_C(1)) = H^1(C, \mathcal{O}_C(2)) = 0.
    \]
\end{proposition}
\begin{proof}
    We prove that $H^1(C, \mathcal{O}_C(1)) = 0$. The argument for proving
    $H^1(C, \mathcal{O}_C(2)) = 0$ is analogous; the proof we present here for $H^1(C, \mathcal{O}_C(1))$ carries over almost word-for-word. Let $v$ denote the normalization map 
    \[
    v: \bigsqcup_{i=1}^{\frac{n+1}{2}} R_i \to C.
    \]
    Note the pullback $\mathcal{O}_{R_i}(1)$ of $\mathcal{O}_C(1)$ to rational normal curve $R_i$ is isomorphic to $\mathcal{O}_{\mathbb{P}_\mathbb{C}^1}(n)$. The line bundle $\mathcal{O}_C(1)$ is equivalent to the data of the line bundles $\mathcal{O}_{R_i}(1)$ on each $R_i$ along with identifications of these line bundles over the nodal intersection points $\{ z_{j,k} \}$ where $1 \leq j \leq \frac{n-1}{2}$ and $1 \leq k \leq n+1$. Here, $z_{j,k}$ denotes the $k$-th nodal intersection between $R_j$ and $R_{j+1}$. The identification data consists of nonzero scalars $\{ c_{j,k} \}$ such that $c_{j,k}$ specifies an isomorphism from the fiber of $\mathcal{O}_{R_j}(1)$ to the fiber of $\mathcal{O}_{R_{j+1}}(1)$ over $z_{j,k}$. Pushing and pulling $\mathcal{O}_C(1)$ along the normalization map, we obtain the short exact sequence
    \[
    0 \to \mathcal{O}_C(1) \to v_* v^* \mathcal{O}_C(1) \to Q \to 0.
    \]
    For points $z \not \in \{ z_{j,k} \}$, we have an isomorphism of stalks $\mathcal{O}_C(1)|_z \cong (v_* v^* \mathcal{O}_C(1))|_z$. Over the points $z_{j,k}$ we have short exact sequences
    \[
    0 \to \mathcal{O}_C(1)|_{z_{j,k}} \to (v_* v^* \mathcal{O}_C(1))|_{z_{j,k}} \to Q|_{z_{j,k}} \to 0.
    \]
    So $Q$ is supported only at the points $z_{j,k}$. Since
$(v_* v^* \mathcal{O}_C(1))|_{z_{j,k}} \cong \mathcal{O}_{R_j}(1)|_{z_{j,k}} \oplus \mathcal{O}_{R_{j+1}}(1)|_{z_{j,k}}$, we have 
\[
Q|_{z_{j,k}} \cong \frac{\mathcal{O}_{R_j}(1)|_{z_{j,k}} \oplus \mathcal{O}_{R_{j+1}}(1)|_{z_{j,k}}}{\mathcal{O}_C(1)|_{z_{j,k}}},
\]
where the denominator is really the image of the injective map $\mathcal{O}_C(1)|_{z_{j,k}} \to (v_* v^* \mathcal{O}_C(1))|_{z_{j,k}}$. Note $\mathcal{O}_C(1)|_{z_{j,k}}$ consists of elements $(\eta, \xi) \in \mathcal{O}_{R_j}(1)|_{z_{j,k}} \oplus \mathcal{O}_{R_{j+1}}(1)|_{z_{j,k}}$ such that $c_{j,k} \cdot \eta(z_{j,k}) = \xi(z_{j,k})$. 
The long exact sequence in sheaf cohomology of the push-pull sequence yields 
    \[
    0 \to H^0(C, \mathcal{O}_C(1)) \to \bigoplus_{i=1}^{\frac{n+1}{2}} H^0(\mathbb{P}_\mathbb{C}^1, \mathcal{O}_{\mathbb{P}_\mathbb{C}^1}(n)) \xrightarrow[]{\alpha}  H^0(C, Q) \to H^1(C, \mathcal{O}_C(1)) \to 0.
    \]
    Then $H^1(C, \mathcal{O}_C(1)) = 0$ if and only if the map $\alpha$ is surjective. Note
    \[
    H^0(C, Q) \cong \bigoplus_{1 \leq j \leq \frac{n-1}{2}, 1 \leq k \leq n+1} \frac{ \mathcal{O}_{R_j}(1)|_{z_{j,k}} \oplus \mathcal{O}_{R_{j+1}}(1)|_{z_{j,k}} }{ \mathcal{O}_C(1)|_{z_{j,k}}} \cong \bigoplus_{1 \leq j \leq \frac{n-1}{2}, 1 \leq k \leq n+1} \frac{\mathcal{O}_{R_j}(1)(z_{j,k}) \oplus \mathcal{O}_{R_{j+1}}(1)(z_{j,k})}{\mathcal{O}_C(1)(z_{j,k})}
    \]
    where, for example, $\mathcal{O}_{R_j}(1)(z_{j,k})$ denotes the fiber of the line bundle over the point $z_{j,k}$.
    \par
    To show surjectivity of $\alpha$, let us first focus on the $z_{1,1}$ term. There is a homogeneous degree $n$ polynomial $F_{11}$ of $H^0(R_1, \mathcal{O}_{R_1}(1)) \cong H^0(\mathbb{P}_\mathbb{C}^1, \mathcal{O}_{\mathbb{P}_\mathbb{C}^1}(n))$ which does not vanish at $z_{1,1}$ but vanishes at $z_{1,2}, \cdots, z_{1,n+1}$. Then $\alpha$ maps $(F_{11}, 0, \cdots, 0)$ to  zero everywhere except in 
    \[
    \frac{\mathcal{O}_{R_1}(1)(z_{1,1}) \oplus \mathcal{O}_{R_{2}}(1)(z_{1,1})}{\mathcal{O}_C(1)(z_{1,1})}, 
    \]
    which is a one-dimensional vector space. Thus, we have surjectivity with respect to $z_{1,1}$. An analogous argument can be made for each $z_{j,k}$. Thus, $\alpha$ is surjective. Hence, $H^1(C, \mathcal{O}_C(1)) = 0$. 
\end{proof}

Since a rational normal curve chain lies in $\mathbb{P}_\mathbb{C}^n$, we can think about how the normal bundle of a Veronese variety of dimension $n$ restricts to the rational normal curve chain. In particular, we show the following.

\begin{proposition}\label{RNCChainMPointsVanishing}
    There exist distinct smooth points $p_1, \cdots, p_{\binom{n+2}{2} + n + 1}$ on a rational normal curve chain $C \subset \mathbb{P}_\mathbb{C}^n$ such that 
    \[
    H^0(C, \mathcal{N}_{V_{n,2}/\mathbb{P}_\mathbb{C}^{\binom{n+2}{2} - 1}}|_C \otimes \mathcal{I}_{p_1, \cdots, p_{\binom{n+2}{2} + n+1}}) = 0.
    \]
\end{proposition}
\begin{proof}
    The vector bundle $\mathcal{N}_{V_{n,2}/\mathbb{P}_\mathbb{C}^{\binom{n+2}{2} - 1}}|_C$ is equivalent to the data of the vector bundles $\mathcal{N}_{V_{n,2}/\mathbb{P}_\mathbb{C}^{\binom{n+2}{2} - 1}}|_{R_j}$ and linear isomorphisms $\phi_{jk}$ identifying the fiber of $\mathcal{N}_{V_{n,2}/\mathbb{P}_\mathbb{C}^{\binom{n+2}{2} - 1}}|_{R_j}$ over $z_{j,k}$ with the fiber of $\mathcal{N}_{V_{n,2}/\mathbb{P}_\mathbb{C}^{\binom{n+2}{2} - 1}}|_{R_{j+1}}$ over $z_{j,k}$. Note 
    \[
    \mathcal{N}_{V_{n,2}/\mathbb{P}_\mathbb{C}^{\binom{n+2}{2} - 1}}|_{R_j} \cong \bigoplus_{i=1}^{\frac{n(n+1)}{2}} \mathcal{O}_{\mathbb{P}_\mathbb{C}^1}(2n+2)
    \]
    by \cite[Theorem 4.3]{ray}. Pick any $2n+3$ smooth points on $R_1$, then pick $n+2$ smooth points on each of the other $\frac{n-1}{2}$ rational normal curves. Twisting down $\mathcal{N}_{V_{n,2}/\mathbb{P}_\mathbb{C}^{\binom{n+2}{2} - 1}}|_{R_1}$ by $2n+3$ points and, for $j\neq 1$, twisting down $\mathcal{N}_{V_{n,2}/\mathbb{P}_\mathbb{C}^{\binom{n+2}{2} - 1}}|_{R_j}$ by $n+2$ points yields
    \[
    \mathcal{N}_{V_{n,2}/\mathbb{P}_\mathbb{C}^{\binom{n+2}{2} - 1}}|_{R_1} \otimes \mathcal{I}_{p_1, \cdots, p_{2n+3}} \cong \bigoplus_{i=1}^{\frac{n(n+1)}{2}} \mathcal{O}_{\mathbb{P}_\mathbb{C}^1}(-1) \text{ and }
    \]
    \[
    \mathcal{N}_{V_{n,2}/\mathbb{P}_\mathbb{C}^{\binom{n+2}{2} - 1}}|_{R_j} \otimes \mathcal{I}_{p_{2n+3+(n+2)(j-2) + 1}, \cdots, p_{2n+3 + (n+2)(j-1)}} \cong \bigoplus_{i=1}^{\frac{n(n+1)}{2}} \mathcal{O}_{\mathbb{P}_\mathbb{C}^1}(n), \text{ for } 2 \leq j \leq \frac{n+1}{2}.
    \]
    A global section $s$ of $\mathcal{N}_{V_{n,2}/\mathbb{P}_\mathbb{C}^{\binom{n+2}{2} - 1}}|_C \otimes \mathcal{I}_{p_1, \cdots, p_{\binom{n+2}{2} + n+1}}$ is equivalent to the data of a global section $s_1$ of $ \mathcal{N}_{V_{n,2}/\mathbb{P}_\mathbb{C}^{\binom{n+2}{2} - 1}}|_{R_1} \otimes \mathcal{I}_{p_1, \cdots, p_{2n+3}}$ and global sections $s_j$ of $\mathcal{N}_{V_{n,2}/\mathbb{P}_\mathbb{C}^{\binom{n+2}{2} - 1}}|_{R_1} \otimes \mathcal{I}_{p_1, \cdots, p_{2n+3}}$, for $j \neq 1$, where $s_i$ and $s_{i+1}$ are compatible with the identification $\phi_{ik}$ of fibers over $z_{i,k}$. So given such a global section $s$, note that $s_1$ must vanish completely over $R_1$. Furthermore, since $s_2$ and $s_1$ are compatible over the points $\{ z_{1,k} \}_{1 \leq k \leq n+1}$, we must have that $s_2$ vanishes at the $n+1$ nodal intersections. Since $s_2$ is a tuple of homogeneous polynomials of degree $n$, this forces $s_2$ to vanish everywhere on $R_2$. Again by compatibility, this forces $s_3$ to vanish everywhere on $R_3$, and so on. This implies that all the $s_j$ must vanish, hence $s=0$.
\end{proof}

\begin{remark}
    Proposition \ref{RNCChainMPointsVanishing} illustrates that if one seeks to smooth a degenerate curve to obtain an auxiliary curve satisfying Properties (1) and (2) of Strategy \ref{strategy}, then it is helpful to know how the Veronese normal bundle restricts to the pieces from which the degenerate curve is built out of. Our knowledge of how degree 2 Veronese normal bundles restrict to rational normal curves motivates our construction with only rational normal curves, which is possible when $n$ is odd \cite[Theorem 1.3]{ray}.
\end{remark}

\section{Smoothing the rational normal curve chain}\label{sectionSmoothingRNCChain}

Fix a degree 2 Veronese embedding $v_{n,2}: \mathbb{P}_\mathbb{C}^n \to \mathbb{P}_\mathbb{C}^{\binom{n+2}{2}-1}$, where $n$ is odd and $n \geq 3$. Let $V_{n,2}$ denote the Veronese variety. In this section we deform a rational normal curve chain $C \subset \mathbb{P}^n_\mathbb{C}$ to a smooth curve $\tilde{C} \subset \mathbb{P}^n_\mathbb{C}$ of degree $\frac{n(n+1)}{2}$ and genus $\frac{n(n-1)}{2}$ satisfying the following properties: 
\begin{enumerate}
    \item $\mathcal{O}_{\tilde{C}}(2)$ and all of its square roots are very ample non-special line bundles.
    \item There exist distinct smooth points $\widetilde{p}_1, \cdots, \widetilde{p}_{\binom{n+2}{2} + n + 1}$ on $v_{n,2}(\tilde{C})$ such that 
    \[
    H^0(v_{n,2}(\tilde{C}), \mathcal{N}_{V_{n,2}/\mathbb{P}_\mathbb{C}^{\binom{n+2}{2} - 1}}|_{v_{n,2}(\tilde{C})} \otimes \mathcal{I}_{\widetilde{p_1}, \cdots, \widetilde{p}_{\binom{n+2}{2} + n+1}}) = 0,
    \]
    where $\mathcal{I}_{\widetilde{p_1}, \cdots, \widetilde{p}_{\binom{n+2}{2} + n+1}}$ is the ideal sheaf of the points with respect to $v_{n,2}(\tilde{C})$. 
\end{enumerate}
Such a curve $\tilde{C}$ would be a satisfactory auxiliary curve in the sense of Strategy \ref{strategy}, allowing us to invoke Proposition \ref{globalSectionDimensionSpecification}. In this section, we first address Property (1), resulting in Proposition \ref{finallyFindingAuxiliaryCurve}. Afterwards, we perform an upper semi-continuity argument to address Property (2). Finally, we conclude our proof of interpolation for degree 2 Veroneses of odd dimension, and demonstrate the enumerative lower bound described by Theorem \ref{mainResult}.
\par
Let us first address Property (1). Recall from Section \ref{sectionRNCChain} that a rational normal curve chain $C \subset \mathbb{P}_\mathbb{C}^n$ is of degree $\frac{n(n+1)}{2}$ and arithmetic genus $\frac{n(n-1)}{2}$ and
\[
H^1(C, \mathcal{O}_C(1)) = H^1(C, \mathcal{O}_C(2)) = 0,
\]
by Proposition \ref{RNCChainLineBundleNotSpecial}. Since the rational normal curve chain $C$ is reduced with nodal singularities and $H^1(C, \mathcal{O}_C(1)) = 0$, there exists a flat family $\mathcal{X} \subseteq \mathbb{P}_\mathbb{C}^{n} \times T$ over an integral $\mathbb{C}$-scheme of finite type $T$, such that the special fiber $\mathcal{X}_0$ is $C$ and there are smooth fibers $\mathcal{X}_t$ for some $t \neq 0$ \cite[Proposition 29.9]{hartshorneDeformation}. Since $T$ is integral and Noetherian, all fibers of the family $\mathcal{X}$ share the same Hilbert polynomial \cite[Chapter 3, Proposition 9.9]{hartshorne}. This implies that we can deform $C$ to a smooth curve in $\mathbb{P}_\mathbb{C}^n$
of degree $n(n+1)$ and genus $\frac{n(n-1)}{2}$. 
\par
Since the projection $\mathcal{X} \to T$ is a proper flat morphism of finite presentation and $T$ is irreducible, the existence of a smooth fiber implies that the projection is smooth over a dense open subset of $T$ by properness and \cite[IV-4-12.2.4]{ega4}. Thus, there is a Zariski open subset of the base $T$, over which the fibers of $\mathcal{X}$ are smooth. Note the function $\dim H^1(\mathcal{X}_t, \mathcal{O}_{\mathcal{X}_t}(1))$ is an upper semi-continuous function of $t \in T$ \cite[Chapter 3, Theorem 12.8]{hartshorne}. Then there is a Zariski open subset $T^o_1 \subseteq T$ over which the fibers $\mathcal{X}_t$ are smooth curves of degree $\frac{n(n+1)}{2}$ and genus $\frac{n(n-1)}{2}$, embedded in $\mathbb{P}_\mathbb{C}^{n}$ by the complete linear system of a non-special very ample line bundle $\mathcal{O}_{\mathcal{X}_t}(1)$. We can consider our family $\mathcal{X}$ as a family of curves in $\mathbb{P}^{\binom{n+2}{2} - 1}$ by the following embedding over $T$: 
\[
\mathcal{X} \to \mathbb{P}^n_\mathbb{C} \times T \xrightarrow[]{(v_{n,2} \times \text{id})} \mathbb{P}^{\binom{n+2}{2} - 1} \times T.
\]
Again, by upper semi-continuity, we have $H^1(\mathcal{X}_t, \mathcal{O}_{\mathcal{X}_t}(2)) = 0$ for $t$ in a Zariski open subset $T^o_2 \subseteq T$. This argument shows that there is a Zariski open subset $T^o \subseteq T$, such that every fiber of the family $\mathcal{X}|_{T^o} \to T^o$ over a $\mathbb{C}$-point $t \in T^o$ is a smooth curve $\mathcal{X}_t \subset \mathbb{P}_\mathbb{C}^n$ of degree $\frac{n(n+1)}{2}$ and genus $\frac{n(n-1)}{2}$, whose very ample line bundles $\mathcal{O}_{\mathcal{X}_t}(1)$ and $\mathcal{O}_{\mathcal{X}_t}(2)$ are non-special.
\par
We are not quite done yet. We need to find a smooth curve $C'$ such that all square roots of $\mathcal{O}_{C'}(2)$ are very ample non-special line bundles. At the moment, we have demonstrated existence of curves $\mathcal{X}_t$ such that just one of their square roots, namely $\mathcal{O}_{\mathcal{X}_t}(1)$, is guaranteed to be very ample and non-special.

\begin{proposition}\label{finallyFindingAuxiliaryCurve}
    There exists a smooth curve $C' \subset \mathbb{P}_\mathbb{C}^n$ of degree $\frac{n(n+1)}{2}$ and genus $\frac{n(n-1)}{2}$, such that 
    $\mathcal{O}_{C'}(2)$ and all of its square roots are very ample non-special line bundles. 
\end{proposition}
\begin{proof}
Let $\mathcal{H}_{\frac{n(n+1)}{2}, \frac{n(n-1)}{2}, n}$ denote the Hilbert scheme of smooth, irreducible, and non-degenerate curves of degree $\frac{n(n+1)}{2}$ and genus $\frac{n(n-1)}{2}$ in $\mathbb{P}^n_\mathbb{C}$. We denote by $\mathcal{H}_{\text{NS}}$ the union of those components of $\mathcal{H}_{\frac{n(n+1)}{2}, \frac{n(n-1)}{2}, n}$ whose general element is linearly normal. Since $n + \frac{n(n-1)}{2} = \frac{n(n+1)}{2}$, the general element of $\mathcal{H}_{\text{NS}}$ is non-special and $\mathcal{H}_{\text{NS}}$ is irreducible \cite[page 1103]{keem}. The family $\mathcal{X}|_{T^o} \to T^o$ is obtained by pulling back along a morphism $T^o \to \mathcal{H}_{\text{NS}}$. 
    \par
    Let $M_{\frac{n(n-1)}{2}}$ denote the moduli space of smooth curves of genus $\frac{n(n-1)}{2}$. Let $\Pic^d$ denote the universal Picard scheme of degree $d$ over $M_{\frac{n(n-1)}{2}}$, which projects onto $M_{\frac{n(n-1)}{2}}$ so that its fiber over a $\mathbb{C}$-point $[X] \in M_{\frac{n(n-1)}{2}}$ is the Picard variety $\Pic^d(X)$ of degree $d$ line bundles on $X$. In particular, there is a universal squaring map between the universal Picard schemes of degrees $\frac{n(n+1)}{2}$ and $n(n+1):$ 
\[\begin{tikzcd}
	{\Pic^{\frac{n(n+1)}{2}}} && {\Pic^{n(n+1)}} \\
	& {M_{\frac{n(n-1)}{2}}}
	\arrow["\text{sq}", from=1-1, to=1-3]
	\arrow[from=1-1, to=2-2]
	\arrow[from=1-3, to=2-2]
\end{tikzcd}\]
    Note the locus of non-special line bundles of degree $n(n+1)$ is a Zariski open subset of $\Pic^{n(n+1)}$. The locus of $\Pic^{n(n+1)}$ which parameterizes pairs $(C,L)$, where $C$ is a genus $\frac{n(n-1)}{2}$ curve and $L$ is a line bundle of degree $n(n+1)$ whose square roots are all non-special, is also Zariski open. By irreducibility of $\Pic^{n(n+1)}$, their intersection, which we call $U_{\text{NS}, \sqrt{\text{NS}}}$, is a Zariski open of $\Pic^{n(n+1)}$ parameterizing pairs $(C,L)$ where $L$ is a non-special line bundle of degree $n(n+1)$ on $C$ whose square roots are all non-special. Then $\text{sq}^{-1}(U_{\text{NS}, \sqrt{\text{NS}}})$ is a Zariski open subset of $\Pic^{\frac{n(n+1)}{2}}$. 
    \par
    Next, consider the natural map $\pi_L: 
    \mathcal{H}_{\text{NS}} \to \Pic^{\frac{n(n+1)}{2}}$
    which sends $\mathbb{C}$-points $[C]$ to pairs $(C, \mathcal{O}_C(1))$. The map $\pi_L$ factors through both the locus of non-special line bundles and the locus of very ample line bundles of $\Pic^{\frac{n(n+1)}{2}}$, both of which are Zariski opens. Their intersection is a Zariski open, which must also intersect $\text{sq}^{-1}(U_{\text{NS}, \sqrt{\text{NS}}})$ by irreducibility of $\Pic^{\frac{n(n+1)}{2}}$. Altogether, this implies that $U := \pi_L^{-1}(\text{sq}^{-1}(U_{\text{NS}, \sqrt{\text{NS}}}))$ is a Zariski open subset of $\mathcal{H}_{\text{NS}}$. 
    \par
    Finally, since $T^o \to \mathcal{H}_{\text{NS}}$ is a morphism whose domain is integral, it factors through the reduced locus 
    \[
    T^o \to  \mathcal{H}^{\text{red}}_{\text{NS}} \to \mathcal{H}_{\text{NS}}
    \]
    by \cite[lemma 26.12.7]{stacks-project}. Note that $U \cap \mathcal{H}^{\text{red}}_{\text{NS}} $ is nonempty. Then since the reduced locus $\mathcal{H}^{\text{red}}_{\text{NS}}$ is irreducible, we can utilize it as a base space over which we can deform a $\mathcal{X}_t$, for some $\mathbb{C}$-point $t \in T^o$, to a smooth curve $C' \subset \mathbb{P}_\mathbb{C}^n$ of degree $\frac{n(n+1)}{2}$ and genus $\frac{n(n-1)}{2}$, such that $\mathcal{O}_{C'}(2)$ and all of its square roots are very ample non-special line bundles. 
\end{proof}

Proposition \ref{finallyFindingAuxiliaryCurve} shows that we can find a smooth curve $C'$ satisfying Property (1). We now address Property (2). Recall, by Proposition \ref{RNCChainMPointsVanishing}, there exist smooth points $p_1, \cdots, p_{\binom{n+2}{2} + n + 1}$ on a rational normal curve chain $C$ such that 
\[
H^0(C, \mathcal{N}_{V_{n,2}/\mathbb{P}_\mathbb{C}^{\binom{n+2}{2} - 1}}|_C \otimes \mathcal{I}_{p_1, \cdots, p_{\binom{n+2}{2} + n+1}}) = 0.
\]
First, let $\pi^*_{\mathbb{P}^n_\mathbb{C}}(\mathcal{N}_{V_{n,2}/\mathbb{P}^{\binom{n+2}{2} - 1}_\mathbb{C}})$ denote the pullback of the Veronese normal bundle along the projection 
\[
\pi_{\mathbb{P}^n_\mathbb{C}}: \mathbb{P}^n_\mathbb{C} \times T \to \mathbb{P}^n_\mathbb{C}.
\]
Letting $\pi^*_{\mathbb{P}^n_\mathbb{C}}(\mathcal{N}_{V_{n,2}/\mathbb{P}^{\binom{n+2}{2} - 1}_\mathbb{C}})|_{\mathcal{X}}$ denote its restriction 
to the family $\mathcal{X}$, note that its further restriction to $\mathcal{X}_t$ is isomorphic to 
\[
\mathcal{N}_{V_{n,2}/\mathbb{P}^{\binom{n+2}{2} - 1}_\mathbb{C}}|_{v_{n,2}(\mathcal{X}_t)},
\]
for every $t \in T$. Second, the points $p_1, \cdots, p_{\binom{n+2}{2} + n + 1}$ of the special fiber $\mathcal{X}_0$ deform to distinct smooth points $p_1^t, \cdots, p^t_{\binom{n+2}{2} + n + 1}$ on $\mathcal{X}_t$, for $t \in T^p$ where $T^p$ is a Zariski open of $T$ containing $0$. In particular, for $t \in T^p$, the line bundles 
\[
\mathcal{O}_{\mathcal{X}_t}(-p^t_1 - \cdots - p^t_{\binom{n+2}{2} + n + 1}) \cong \mathcal{I}_{p_1^t, \cdots, p^t_{\binom{n+2}{2} + n + 1}} 
\]
are isomorphic to the restrictions of a global line bundle $\mathcal{L}$ on $\mathcal{X}|_{T^p}$ to $\mathcal{X}_t$. Considering the vector bundle 
\[
\pi^*_{\mathbb{P}^n_\mathbb{C}}(\mathcal{N}_{V_{n,2}/\mathbb{P}^{\binom{n+2}{2} - 1}_\mathbb{C}})|_{(\mathcal{X}|_{T_p})} \otimes \mathcal{L}
\]
on $\mathcal{X}|_{T^p}$, we have by upper semi-continuity \cite[Chapter 3, Theorem 12.8]{hartshorne} that 
\begin{equation}\label{finalVanishingWeWant}
    H^0(v_{n,2}(\mathcal{X}_t), \mathcal{N}_{V_{n,2}/\mathbb{P}_\mathbb{C}^{\binom{n+2}{2} - 1}}|_{v_{n,2}(\mathcal{X}_t)} \otimes \mathcal{I}_{p^t_1, \cdots, p^t_{\binom{n+2}{2} + n+1}}) = 0
\end{equation}
for general $t \in T^p$. Considering the intersection $T^p$ with $T^o$, we have in summary: a Zariski open subset $T^{\text{final}} \subseteq T^o$ such that for $t \in T^{\text{final}} $, the fiber $\mathcal{X}_t \subset \mathbb{P}^n_\mathbb{C}$ is a smooth curve of degree $\frac{n(n+1)}{2}$ and genus $\frac{n(n-1)}{2}$ whose very ample line bundles $\mathcal{O}_{\mathcal{X}_t}(1)$ and $\mathcal{O}_{\mathcal{X}_t}(2)$ are non-special, and
\[
H^0(v_{n,2}(\mathcal{X}_t), \mathcal{N}_{V_{n,2}/\mathbb{P}^{\binom{n+2}{2} - 1}_\mathbb{C}}|_{v_{n,2}(\mathcal{X}_t)} \otimes \mathcal{I}_{p_1^t, \cdots, p^t_{\binom{n+2}{2} + n + 1}}) = 0.
\]

Finally, using the proof of Proposition \ref{finallyFindingAuxiliaryCurve}, we have a map $T^{\text{final}} \to \mathcal{H}^{\text{red}}_{NS}$. Applying the same upper semi-continuity argument but this time with $\mathcal{H}^{\text{red}}_{NS}$ as the base, we conclude that there does indeed exist smooth curves $\widetilde{C}$ with distinct points $\widetilde{p_1}, \cdots, \widetilde{p}_{\binom{n+2}{2} + n + 1}$ on $v_{n,2}(\tilde{C})$ such that
\begin{equation}\label{vanishAtPointsImplyVanishOnCurve}
    H^0(v_{n,2}(\tilde{C}), \mathcal{N}_{V_{n,2}/\mathbb{P}_\mathbb{C}^{\binom{n+2}{2} - 1}}|_{v_{n,2}(\tilde{C})} \otimes \mathcal{I}_{\widetilde{p_1}, \cdots, \widetilde{p}_{\binom{n+2}{2} + n+1}}) = 0,
\end{equation}
and $\mathcal{O}_{\tilde{C}}(2)$ and all of its square roots are very ample non-special line bundles.
\par
Now that we have found an auxiliary curve in the sense of Strategy \ref{strategy}, we conclude our proof of interpolation. First, equation \ref{vanishAtPointsImplyVanishOnCurve} implies 
\begin{equation}\label{vanishAtPointsVanishOnCurveStep2}
    H^0(V_{n,2}, \mathcal{N}_{V_{n,2}/\mathbb{P}_\mathbb{C}^{\binom{n+2}{2} -1}} \otimes \mathcal{I}_{\widetilde{p_1}, \cdots, \widetilde{p}_{\binom{n+2}{2} + n+1}}) = H^0(V_{n,2}, \mathcal{N}_{V_{n,2}/\mathbb{P}_\mathbb{C}^{\binom{n+2}{2} -1}} \otimes \mathcal{I}_{v_{n,2}(\tilde{C})}).
\end{equation}
Applying the Riemann-Roch theorem to $\mathcal{O}_{\tilde{C}}(2)$ and all of its square roots implies that the dimensions of their spaces of global sections are $\binom{n+2}{2}$ and $n+1$, respectively. Thus, Proposition 
\ref{globalSectionDimensionSpecification} holds for $v_{n,2}(\tilde{C})$, which implies that
\begin{equation}\label{finalStep}
    H^0(V_{n,2}, \mathcal{N}_{V_{n,2}/\mathbb{P}_\mathbb{C}^{\binom{n+2}{2} - 1}} \otimes \mathcal{I}_{v_{n,2}(\tilde{C})}) = 0,
\end{equation}
by the discussion in Subsection \ref{subsectionStrategy}. By Proposition \ref{goal}, equations \ref{vanishAtPointsVanishOnCurveStep2} and \ref{finalStep} imply that degree 2 Veronese varieties of odd dimension satisfy interpolation.
\par
Finally, we show the enumerative lower bound described in Theorem \ref{mainResult}. Consider the projection map in the incidence correspondence of Definition \ref{interpolationDefinition}:
\[
\Psi \to \mathcal{H}_{V_{n,2}} \times (\prod_{i=1}^{\binom{n+2}{2} + n + 1} \Gr(1, \binom{n+2}{2}) ) \times (\mathbb{P}_\mathbb{C}^{\binom{n+2}{2} - 1})^{\binom{n+2}{2} + n + 1} \to (\prod_{i=1}^{\binom{n+2}{2} + n + 1} \Gr(1, \binom{n+2}{2}) ).
\]
Interpolation implies that there is a dense open subset of $(\prod_{i=1}^{\binom{n+2}{2} + n + 1} \Gr(1, \binom{n+2}{2}) )$, over which this projection has geometric fibers which are nonempty and zero-dimensional \cite[Appendix]{landesmanDelPezzo}. Furthermore, an auxiliary curve $v_{n,2}(\tilde{C})$ satisfies Proposition \ref{globalSectionDimensionSpecification}. Since $v_{n,2}(\tilde{C})$ is smooth, there are $2^{n(n-1)}$ square roots of $\mathcal{O}_{v_{n,2}(\tilde{C})}(1) \cong \mathcal{O}_{\tilde{C}}(2)$. Thus, there are $2^{n(n-1)}$ degree 2 Veroneses of odd dimension which contain $v_{n,2}(\tilde{C})$. This demonstrates that over a closed point of $(\prod_{i=1}^{\binom{n+2}{2} + n + 1} \Gr(1, \binom{n+2}{2}) )$ corresponding to $(\widetilde{p_1}, \cdots, \widetilde{p}_{\binom{n+2}{2} + n + 1})$, the zero-dimensional fiber of this projection has at least $2^{n(n-1)}$ points. By \cite[Theorem 4.17(iii)]{lowerSemi}, the number of connected components of the geometric fibers of this projection map is a lower semicontinuous function. This implies that there is a dense open subset of $(\prod_{i=1}^{\binom{n+2}{2} + n + 1} \Gr(1, \binom{n+2}{2}) )$, over which the geometric fibers have at least $2^{n(n-1)}$ connected components. Thus, there is a dense open subset of $(\prod_{i=1}^{\binom{n+2}{2} + n + 1} \Gr(1, \binom{n+2}{2}) )$ over which the projection map has zero-dimensional fibers comprised of at least $2^{n(n-1)}$ points. This completes the proof of Theorem \ref{mainResult}.

\section{Concluding Remarks and Questions}\label{sectionConclusion}

In summary, we proved interpolation for degree 2 Veroneses of odd dimension by showing that certain normal vector fields vanish. We achieved this vanishing by utilizing certain auxiliary curves, which we found by smoothing rational normal curve chains. 

\begin{question}
Aside from surfaces, do degree 2 Veronese varieties of even dimension satisfy interpolation? 
\end{question}

The author believes the answer is likely yes, but in the case of even dimensions, rational normal curve chains will not have the desired degree and genus in the sense of Subsection \ref{strategy}. An auxiliary curve for the case of even dimensions would need to be constructed from curves beyond rational normal curves. Thus, knowledge of how degree 2 Veronese normal bundles restrict to other curves may lead to progress in this direction. We note that, in general, Veronese normal bundles restricted to rational normal curves are not well-balanced since, for degree $d$ Veroneses of dimension $n$, the degree of the first Chern class of the restriction is $\binom{n+d}{d}nd - n(n+1)$ and the rank is $\binom{n+d}{d} - n - 1$. However, besides the degree 2 case, the numerology of the first Chern class and rank of the restriction does work out in the case of degree 3 Veroneses of dimension 7. 

\begin{question}\label{deg3dim7question}
Let $\mathbb{P}_\mathbb{C}^7 \to \mathbb{P}_\mathbb{C}^{\binom{10}{3} - 1}$ be an embedding for a degree 3 Veronese of dimension 7. Let $R$ denote a rational normal curve in $\mathbb{P}_\mathbb{C}^7$. In this case, is the Veronese normal bundle restricted to the rational normal curve well-balanced? In other words, is
\[
\mathcal{N}_{V_{n,2}/\mathbb{P}_\mathbb{C}^{\binom{10}{3} - 1}}|_{R} \cong \bigoplus_{i=1}^{112} \mathcal{O}_{\mathbb{P}_\mathbb{C}^1}(22)?
\]
\end{question}

If the answer to question \ref{deg3dim7question} is yes, then we can deduce interpolation for degree 3 Veroneses of dimension 7 by, essentially word-for-word, the same argument presented in this paper. Note, however, that the well-balancedness of the restriction to rational normal curves is not a prerequisite for interpolation. For example, degree 3 Veronese surfaces satisfy interpolation \cite[Theorem 6.1]{landesmanDelPezzo} but the numerology of their normal bundles restricted to rational normal curves in $\mathbb{P}_\mathbb{C}^2$ immediately implies that they are not well-balanced.

\section{Acknowledgements}

The author thanks Anand Patel for introducing them to the questions that led to this paper, and thanks Anand Patel and Joe Harris for their invaluable mentorship, instruction, and encouragement during the research process. Without them, this project would not have been possible. The author is grateful to Aaron Landesman for helpful conversations and correspondences. The author was supported by Harvard's PRISE fellowship for a significant portion of their research time.

\bibliographystyle{alpha}
 \bibliography{bibliography}

\begin{thebibliography}{{Sta}24}

\bibitem[ALY19]{atansovLarsonYang}
Atanas Atanasov, Eric Larson, and David Yang.
\newblock Interpolation for normal bundles of general curves.
\newblock {\em Memoirs of the American Mathematical Society}, 257, 2019.

\bibitem[Ata15]{atansov}
Atanas Atanasov.
\newblock Interpolation and vector bundles on curves.
\newblock 2015.

\bibitem[Bal14]{ballico}
Edoardo Ballico.
\newblock An interpolation problem for the normal bundle of curves of genus $g\ge 2$ and high degree in $\mathbb {P}^r$.
\newblock {\em Communications in Algebra}, 45, 04 2014.

\bibitem[Cob22]{coble}
Arthur~B. Coble.
\newblock Associated sets of points.
\newblock {\em Transactions of the American Mathematical Society}, 24(1):1--20, 1922.

\bibitem[Cos06]{coskun}
Izzet Coskun.
\newblock Degenerations of surface scrolls and the gromov-witten invariants of grassmannians.
\newblock {\em Journal of Algebraic Geometry}, 15(2):233--284, 2006.

\bibitem[Cra50]{cramer}
G.~Cramer.
\newblock Introduction à l’analyse des lignes courbes algébriques,.
\newblock 1750.

\bibitem[DM69]{lowerSemi}
Pierre Deligne and David Mumford.
\newblock The irreducibility of the space of curves of given genus.
\newblock {\em Publications Mathématiques de l’Institut des Hautes Scientifiques}, 36:75--109, 1969.

\bibitem[EH84]{ellingsrudhirschowitz}
Geir Ellingsrud and Andre Hirscowitz.
\newblock Sur le fibre normal des courbes gauches.
\newblock {\em C.R. Acade. Sci. Paris Ser. I Math.}, 299(7):245--248, 1984.

\bibitem[EP00]{eisenbudPopescu}
David Eisenbud and Sorin Popescu.
\newblock The projective geometry of the gale transform.
\newblock {\em Journal of Algebra}, 230(1):127--173, 2000.

\bibitem[Gro67]{ega4}
Alexander Grothendieck.
\newblock Éléments de géométrie algébrique : Iv. Étude locale des schémas et des morphismes de schémas, quatrième partie.
\newblock {\em Publications Mathématiques de l'IHÉS}, 32:5--361, 1967.

\bibitem[Har77]{hartshorne}
Robin Hartshorne.
\newblock {\em Algebraic Geometry}.
\newblock Springer, 1977.

\bibitem[Har10]{hartshorneDeformation}
Robin Hartshorne.
\newblock {\em Deformation theory}.
\newblock Springer, 2010.

\bibitem[Kee22]{keem}
Changho Keem.
\newblock On the hilbert scheme of linearly normal curves in $\mathbb{P}^r$ with small index of speciality.
\newblock {\em Indagationes Mathematicae}, 33(5):1102--1124, 2022.

\bibitem[Lan16]{landesmanMinimalDegree}
Aaron Landesman.
\newblock Interpolation of varieties of minimal degree.
\newblock {\em International Mathematics Research Notices}, 2018, 2016.

\bibitem[Lar17]{maximalrank}
Eric Larson.
\newblock The maximal rank conjecture.
\newblock 2017.

\bibitem[LP16]{landesmanDelPezzo}
Aaron Landesman and Anand Patel.
\newblock Interpolation problems: Del pezzo surfaces.
\newblock {\em Annali Scuola Normale Superiore - Classe di Scienze}, 19, 2016.

\bibitem[LV21]{larsonVogt2}
Eic Larson and Isabel Vogt.
\newblock Interpolation for brill-noether curves in $\mathbb{P}^4$.
\newblock {\em European Journal of Mathematics}, 7(1):235--271, 2021.

\bibitem[LV23]{larsonVogt}
Eric Larson and Isabel Vogt.
\newblock Interpolation for brill-noether curves.
\newblock {\em Forum of Mathematics}, 2023.

\bibitem[Per87]{perrin}
D.~Perrin.
\newblock Courbes passant par m points generaux de $\mathbb{P}^3$.
\newblock {\em Mem. Soc. Math. France}, 1987.

\bibitem[Ran07]{ran}
Ziv Ran.
\newblock Normal bundles of rational curves in projective spaces.
\newblock {\em Asian Journal of Mathematics}, 11(4):567--608, 2007.

\bibitem[RS60]{reedSolomon}
I.S. Reed and G.~Solomon.
\newblock Polynomial codes over certain finite fields.
\newblock {\em Journal of the Society for Industrial and Applied Mathematics}, 8(2):300--304, 1960.

\bibitem[Sac81]{sacchiero}
Gianni Sacchiero.
\newblock Normal bundles of rational curves in projective space.
\newblock {\em Ann. Univ. Ferrar Sez VII}, 26:33--40, 1981.

\bibitem[Ser06]{sernesi}
Edoardo Sernesi.
\newblock {\em Deformations of Algebraic Schemes}.
\newblock Springer, 2006.

\bibitem[Sha79]{shamir}
Adi Shamir.
\newblock How to share a secret.
\newblock {\em ACM}, 22(11):612--613, 1979.

\bibitem[Sha24]{ray}
Ray Shang.
\newblock Slope semistability for veronese normal bundles.
\newblock 2024.

\bibitem[{Sta}24]{stacks-project}
The {Stacks project authors}.
\newblock The stacks project.
\newblock \url{https://stacks.math.columbia.edu}, 2024.

\bibitem[Ste89]{stevens}
Jan Stevens.
\newblock On the number of points determining a canonical curve.
\newblock {\em Nederl. Akad. Wetensch. Indaj. Math}, 51(4):485--494, 1989.

\bibitem[Ste96]{stevens2}
Jan Stevens.
\newblock On the computation of versal deformations.
\newblock {\em Nederl. Akad. Wetensch. Indaj. Math}, 82:3717--3720, 1996.

\bibitem[Vog18]{vogt}
Isabel Vogt.
\newblock Interpolation for brill-noether space curves.
\newblock {\em Manuscripta Math}, 156:137--147, 2018.

\bibitem[War79]{waring}
E.~Waring.
\newblock Problems concerning interpolations.
\newblock {\em Philosophical Transactions of the Royal Society}, 69:59--67, 1779.

\end{thebibliography}

\end{document}